\documentclass[a4paper,10pt]{article}
\usepackage{amssymb,amsmath,amsthm}
\usepackage{hyperref}
\usepackage{eucal}
\usepackage{graphicx}
\newcommand{\ie}{\emph{i.e.}}

\newcommand{\cf}{\emph{cf.}}
\newcommand{\Real}{\mathbb{R}}
\newcommand{\Nat}{\mathbb{N}}

\newcommand{\dist}{\mathop{\mathrm{dist}}\nolimits}

\newcommand{\sii}{L^2}
\newtheorem{Theorem}{Theorem}

\theoremstyle{definition}

\begin{document}
%
\title{\textbf{\large Waveguides with asymptotically diverging twisting}}
\author{David Krej\v{c}i\v{r}\'ik}
\date{\small 
\emph{Department of Theoretical Physics,
Nuclear Physics Institute ASCR, \\
25068 \v{R}e\v{z}, Czech Republic; 
krejcirik@ujf.cas.cz}
\medskip \\
15 July 2014}
\maketitle
\begin{abstract}
\noindent
We provide a class of unbounded three-dimensional domains of infinite volume
for which the spectrum of the associated Dirichlet Laplacian is purely discrete.
The construction is based on considering tubes 
with asymptotically diverging twisting angle.
\end{abstract}
%
%
\section{Introduction}
%
Given an open set~$\Omega$ in~$\Real^d$, with $d \geq 1$,
we denote by $-\Delta_D^\Omega$ the Dirichlet Laplacian in $\sii(\Omega)$.
The significant features of~$\Omega$ as regards the spectrum of $-\Delta_D^\Omega$
are conveniently described by means of the classification introduced by
Glazman~\cite[Sec.~49]{Glazman} (see also \cite[Sec.~X.6.1]{Edmunds-Evans}):
$\Omega$~is \emph{quasi-conical} if it contains arbitrarily large balls;
$\Omega$~is \emph{quasi-cylindrical} if it is not quasi-conical
but contains an infinite set of identical disjoint balls;
$\Omega$~is \emph{quasi-bounded} if it is neither 
quasi-conical nor quasi-cylindrical.

By choosing a singular sequence of functions supported in the enlarging balls,
it is easy to see (\cf~\cite[Thm.~X.6.5]{Edmunds-Evans})
that the spectrum of $-\Delta_D^\Omega$ 
is purely essential whenever~$\Omega$ is quasi-conical,
namely, 
$
  \sigma(-\Delta_D^\Omega)
  = \sigma_\mathrm{ess}(-\Delta_D^\Omega)
  = [0,\infty)
$. 
The other extreme case is given by quasi-bounded~$\Omega$,
for which the spectrum is ``typically'' purely discrete,
\ie\ 
$
  \sigma_\mathrm{ess}(-\Delta_D^\Omega)
  = \varnothing
$,
at least under some regularity assumptions about~$\Omega$
(precise conditions can be stated in terms of capacity,
\cf~\cite[Thm.~VIII.3.1]{Edmunds-Evans}).
These regularity assumptions are automatically satisfied 
if~$\Omega$ is bounded or, more generally, of finite volume.  
In general,
it is well known that the discreteness of the spectrum of $-\Delta_D^\Omega$ 
is equivalent to the compactness of the resolvent, which in turn is
equivalent to the compactness of the embedding
\begin{equation}\label{embed}
  W_0^{1,2}(\Omega) \hookrightarrow \sii(\Omega) 
  \,.
\end{equation}
The spectral picture is most complicated for quasi-cylindrical domains,
where both essential and discrete spectra might be present.
It is not difficult to see that 
$
  \sigma_\mathrm{ess}(-\Delta_D^\Omega)
  \not= \varnothing
$
whenever~$\Omega$ is quasi-cylindrical,
while the existence of discrete eigenvalues is considered
as a non-trivial property here. 

A distinguished class of quasi-cylindrical domains is 
represented by unbounded \emph{tubes}, \ie\ tubular neighbourhoods 
of unbounded submanifolds in~$\Real^d$.
The motivation for this restriction is at least two-fold.
First, the tubular geometry is rich enough to demonstrate
the mathematical complexity of the class of quasi-cylindrical domains.
Second, the Dirichlet Laplacian in tubes is
a reasonable model for the Hamiltonian
in quantum-waveguide nanostructures,
where the unboundedness is required to have non-trivial transport properties.
The latter has lead to an enormous number of research papers in recent years,
in which the interplay between the geometry of a tube 
and the spectrum of the associated Dirichlet Laplacian 
is investigated in various settings;
we especially refer to the survey articles
\cite{DE,KKriz,K6-with-erratum,Wachsmuth-Teufel,Haag-Lampart-Teufel_2014}
with many references.  
As the most recent result, let us point out the repulsive effect
of \emph{twisting} in three-dimensional tubes about a spatial curve,
which can be mathematically formalised in terms of Hardy-type inequalities
\cite{EKK,K6-with-erratum,KZ1,BHK}.

The objective of the present note is to show that,
apart from expectable quasi-cylindrical settings, 
the tubular geometry gives rise to an interesting class 
of quasi-bounded domains with purely discrete spectrum, too.
To the best of our knowledge, this class of domains
has not been considered yet in connection with 
the compactness property of the Sobolev embedding~\eqref{embed}.
Our construction is based on considering tubes 
with a twisting angle diverging at the infinity of~$\Real^d$.
For simplicity, we restrict to $d=3$ in this note.

We would like to emphasise that the waveguides considered
in this note have a \emph{uniform cross-section}.
This makes the model very different from 
the recent study~\cite{Exner-Barseghyan_2013},
where the cross-section is allowed to vary along the reference line
and shrinks to a point at the infinity of the waveguide.
The latter leads to cusp-type domains for which compactness
of~\eqref{embed} follows easily by an equivalent
characterisation of quasi-boundedness
\begin{equation}\label{alt}
  \lim_{\stackrel{x\in\Omega}{|x|\to\infty}} \dist(x,\partial\Omega) = 0
  \,.
\end{equation}
In the present model the validity of~\eqref{alt} is not that evident.

\section{Twisted waveguides}
%
A (non-bent) twisted tube~$\Omega$
is obtained by translating and rotating a bounded open 
connected set $\omega \subset \Real^2$ about 
a straight line in~$\Real^3$. 
More precisely, following~\cite{K6-with-erratum}, 
$\Omega$ is defined as the image 
of a straight tube $\Omega_0:=\Real\times\omega$ via the mapping
\begin{equation}\label{diffeomorphism}
  \mathcal{L}(x)
  := \big(
  x_1,
  x_2\cos\theta(x_1)+x_3\sin\theta(x_1),
  -x_2\sin\theta(x_1)+x_3\cos\theta(x_1)
  \big)
  \,,
\end{equation}
\ie~$\Omega := \mathcal{L}(\Omega_0)$.
Here $\theta:\Real\to\Real$ is the rotation angle,
for which we merely assume 
\begin{equation}\label{mere}
  \theta \in W_\mathrm{loc}^{1,\infty}(\Real)
  \,.
\end{equation}
%

In this note, we are interested in spectral properties of
the Dirichlet Laplacian $-\Delta_D^\Omega$ in $\sii(\Omega)$
in the situation when the rotation angle explodes at infinity, \ie,
\begin{equation}\label{explode}
  \lim_{|x_1|\to\infty}|\dot\theta(x_1)|
  = +\infty
  \,.
\end{equation}
This situation is briefly discussed (without any proofs)
in \cite[Sec.~6.2]{KSed}.

Let us now give some intuition for the results 
established below.

\begin{enumerate}
\item
If the tube~$\Omega$ is untwisted 
(\ie~either $\theta$ is constant or $\omega$ is rotationally
invariant with respect to the origin in~$\Real^2$),
then~$\Omega$ can be identified with the straight tube~$\Omega_0$
and the spectrum is easily found by a separation
of variables,
\begin{equation}\label{straight}
  \sigma(-\Delta_D^{\Omega_0}) = [E_1,\infty)
  \,,
\end{equation}
where~$E_1$ denotes the lowest eigenvalue of $-\Delta_D^\omega$
in $\sii(\omega)$. 
\item
If~$\Omega$ is twisted but $\dot\theta$ vanishes at infinity,
then the spectrum~\eqref{straight} is not changed
\cite[Thm.~4.1 \& Sec.~6.1]{K6-with-erratum}.
However, there is a fine difference reflected in the fact
that $-\Delta_D^\Omega$ satisfies a Hardy-type inequality 
\cite{EKK,K6-with-erratum,KZ1}
\begin{equation}\label{Hardy}
  -\Delta_D^{\Omega_0} - E_1 \geq \rho
  \,,
\end{equation}
where~$\rho:\Omega \to [0,\infty)$ is a non-trivial function.
\item
This Hardy inequality turns into a Poincar\'e-type inequality
(\ie\ $\rho$~bound\-ed from below by a positive constant)
if~$|\dot\theta|$ is bounded from below by a positive constant
\cite[Sec.~6.1]{K6-with-erratum}. Consequently,
the spectrum of $-\Delta_D^\Omega$ starts strictly above~$E_1$ in this case.
\item
For instance, if~$\Omega$ is constantly twisted, 
\ie\ $\dot\theta(x_1) = \beta$, 
then the spectral problem
can be solved by a Floquet-type decomposition
and it is shown in \cite{EKov_2005} 
that
\begin{equation}\label{periodic}
  \sigma(-\Delta_D^{\Omega}) = [\lambda_1,\infty)
  \,,
\end{equation}
where $\lambda_1$ is the lowest eigenvalue of 
$-\Delta_D^\omega - \beta^2 \partial_\tau^2$ in $\sii(\omega)$,
with $\partial_\tau:=x_3\partial_2-x_2\partial_3$ being
the transverse angular derivative. 
Note that $\lambda_1>E_1$ if, and only if,
$\beta\not=0$ and $\omega$~is not rotationally
invariant with respect to the origin in~$\Real^2$.
\end{enumerate}

Our objective is to show that the repulsive effect
in the situation~\eqref{explode} is so strong that
the the threshold of the essential spectrum may 
even diverge to infinity, so that the spectrum of $-\Delta_D^{\Omega}$
becomes purely discrete.

\section{Quasi-cylindrical realisation}
%
First, however, let us show that the condition~\eqref{explode} 
by itself is not sufficient to guarantee the absence
of essential spectrum.
\begin{Theorem}\label{Thm.cylindre}
Let $0 \in \omega$.
Then 
$$
  \sigma(-\Delta_D^{\Omega}) \supset [\mu_1,\infty)
  \,,
$$
where~$\mu_1$ denotes the first eigenvalue of 
the Dirichlet Laplacian on the disc~$D_r$ of radius
$r:=\dist(0,\partial\omega)$. 
\end{Theorem}
\begin{proof}
By the definition of~$\Omega$, the cylinder $\Real \times D_r$ 
clearly lies inside~$\Omega$. The spectrum of the Dirichlet Laplacian
in the cylinder coincides with the interval $[\mu_1,\infty)$.
Applying the Weyl criterion for $-\Delta_D^{\Omega}$
with a singular sequence for the cylinder,
we end up with the claim.
\end{proof}

Since 
$
  \Real \times D_r
  \subset \Omega \subset
  \Real \times D_R
$ 
whenever $0 \in \omega$,
where $R:=\sup_{t\in\omega}|t|$,
it follows that~$\Omega$ is quasi-cylindrical in this case.
The geometrical setting of Theorem~\ref{Thm.cylindre}
is illustrated by Figure~\ref{Fig1}.

\begin{figure}[h!]
\begin{center}
\includegraphics[width=0.7\textwidth]{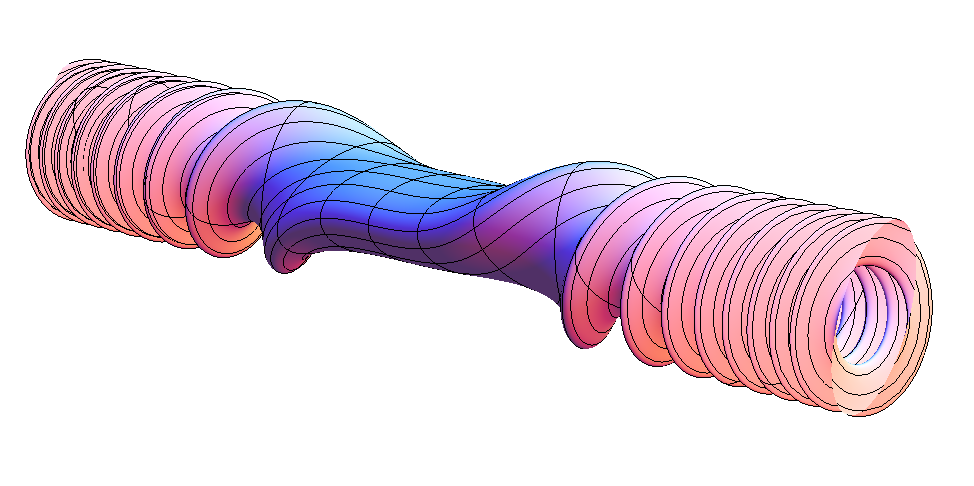}
\end{center}
\caption{Quasi-cylindrical realisation:
Waveguide of elliptical cross-section 
with a centre of rotation in the interior of the ellipse.
The embedded cylindrical channel is responsible 
for the existence of essential spectrum.}
\label{Fig1}
\end{figure}
%

\section{Quasi-bounded realisation}
%
Now we give a sufficient condition which guarantees
the absence of essential spectrum.
\begin{Theorem}\label{Thm.bounded}
Let $\omega \subset \{(t_1,t_2)\in\Real^2 \ | \ t_1 > 0\}$
and assume~\eqref{mere} and~\eqref{explode}.
Then 
$$
  \sigma(-\Delta_D^{\Omega}) = \sigma_\mathrm{disc}(-\Delta_D^{\Omega})
  \,.
$$
\end{Theorem}
\begin{proof}
It follows from~\eqref{explode} that for any $n \in \Nat$
there exists $s_n \in (0,\infty)$ such that 
$|\dot\theta(x_1)| > n$ for every $|x_1| > s_n$.
Fix $n \in \Nat$ and define
$I_\mathrm{int} := (-s_n,s_n)$,
$I_\mathrm{ext} := (-\infty,-s_n) \cup (s_n,+\infty)$.
We also define 
$\Omega_\mathrm{int}:=\mathcal{L}(I_\mathrm{int}\times\omega)$ and 
$\Omega_\mathrm{ext}:=\mathcal{L}(I_\mathrm{ext}\times\omega)$.
Imposing a supplementary Neumann condition at the cuts 
$\{\pm s_n\} \times \omega$, we obtain a lower bound 
$$
  -\Delta_D^{\Omega} \geq 
  -\Delta_{DN}^{\Omega_\mathrm{int}} 
  \oplus -\Delta_{DN}^{\Omega_\mathrm{ext}}
  \,.
$$
Here $-\Delta_{DN}^{\Omega_\mathrm{int}}$ denotes the operator
in $\sii(\Omega_\mathrm{int})$ associated with the quadratic form 
$\psi \mapsto \|\nabla\psi\|_{\sii(\Omega_\mathrm{int})}^2$
with the domain being the restriction of $H_0^1(\Omega)$
to $\Omega_\mathrm{int}$, 
and similarly for $-\Delta_{DN}^{\Omega_\mathrm{ext}}$.
The spectrum of $-\Delta_{DN}^{\Omega_\mathrm{int}}$
is purely discrete. Hence, by the minimax principle,
$$
  \inf\sigma_\mathrm{ess}(-\Delta_D^{\Omega}) 
  \geq \inf\sigma_\mathrm{ess}(-\Delta_{DN}^{\Omega_\mathrm{ext}})
  \geq \inf\sigma(-\Delta_{DN}^{\Omega_\mathrm{ext}})
  \,.
$$

On $\Omega_\mathrm{ext}$ the magnitude of the velocity 
of rotation of the cross-section~$\omega$ about $\Real\times\{0\}$
is greater than~$n$.
Hence, the longitudinal distance on which the cross-section
turns by 180 degrees is estimated from above by~$\pi/n$.
That is, by our hypothesis, the rays on~$\Omega_\mathrm{ext}$
parallel to the axis of rotation 
and emanating from the cuts $\{\pm s_n\} \times \omega$
will hit a Dirichlet boundary of~$\Omega_\mathrm{ext}$ 
at most at the distance~$\pi/n$.
The same argument shows that the rays on~$\Omega_\mathrm{ext}$
parallel to the axis of rotation 
and emanating from any cut $\{x_1\} \times \omega$,
with $|x_1| > s_n+\pi/n$,
will hit a Dirichlet boundary of~$\Omega_\mathrm{ext}$ 
in both directions at most at the distance~$\pi/n$.
Consequently, for any~$\psi$ from the form domain of
$-\Delta_{DN}^{\Omega_\mathrm{ext}}$, we have
$$
  \|\nabla\psi\|_{\sii(\Omega_\mathrm{ext})}^2
  \geq \|\partial_1\psi\|_{\sii(\Omega_\mathrm{ext})}^2
  \geq \left(\frac{\pi}{2\,\pi/n}\right)^2 
  \|\psi\|_{\sii(\Omega_\mathrm{ext})}^2
  \,.
$$ 
By the minimax principle, it thus follows
$$
  \inf\sigma(-\Delta_{DN}^{\Omega_\mathrm{ext}})
  \geq \frac{n^2}{4}
  \,.
$$
Since~$n$ can be chosen arbitrarily large,
we conclude with 
$\sigma_\mathrm{ess}(-\Delta_D^{\Omega}) = \varnothing$.
\end{proof}

It follows from Theorem~\ref{Thm.bounded} that
the embedding~\eqref{embed} is compact.
Since the quasi-boundedness is a necessary condition 
for the compactness of~\eqref{embed},
we see that~$\Omega$ is quasi-bounded under the hypothesis 
of Theorem~\ref{Thm.bounded}.
The geometrical setting of Theorem~\ref{Thm.bounded}
is illustrated by Figure~\ref{Fig2}.

\begin{figure}[h!]
\begin{center}
\includegraphics[width=0.7\textwidth]{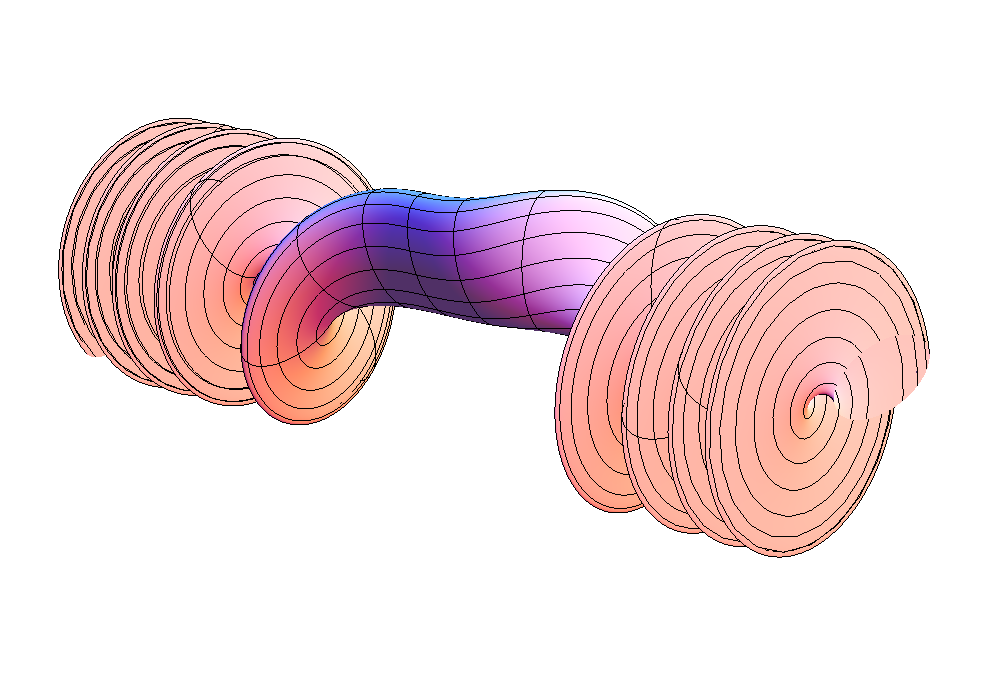}
\end{center}
\caption{Quasi-bounded realisation:
Waveguide of elliptical cross-section 
with a centre of rotation in the exterior of the ellipse.}
\label{Fig2}
\end{figure}
%

\section{Conclusion}
%
Condition~\eqref{explode} leads to a new class of waveguides,
whose detailed spectral analysis constitutes an interesting open problem.

For instance, in the quasi-cylindrical setting (Theorem~\ref{Thm.cylindre}),
one could study the existence of discrete spectra or Hardy-type inequalities
and the rate of accumulation of discrete eigenvalues 
to the threshold of the essential spectrum.

In the quasi-bounded setting (Theorem~\ref{Thm.bounded}),
an interesting question is to establish 
some sort of Weyl-type asymptotics for the discrete eigenvalues.
Note that~$\Omega$ has an infinite volume 
regardless of the choice of~$\theta$.
In fact, the volume is ``locally preserved'' for twisted tubes,
\ie\ 
$
  |\mathcal{L}\big((a_1,a_2)\times\omega\big)|
  = |(a_1,a_2)\times\omega|
$
for any $a_1 < a_2$. 
This makes the present model very different from~\cite{Exner-Barseghyan_2013},
where the cross-section shrinks to a point at the infinity of the waveguide. 
Sharp semiclassical estimates of eigenvalues 
in distinct types of domains of infinite volume 
has been established recently in \cite{Geisinger-Weidl_2011}.

\subsection*{Acknowledgment}
%
I would like to thank Rupert L.~Frank
who raised the question of waveguides with asymptotically diverging twisting
after my talk in the \emph{Isaac Newton Institute for Mathematical Sciences}
in Cambridge, April 2007.
The research was partially supported
by the project RVO61389005 and the GACR grant No.\ 14-06818S.
The author also acknowledges the award from 
the \emph{Neuron fund for support of science},
Czech Republic, May 2014.

%
{\small
\bibliography{bib}

\providecommand{\bysame}{\leavevmode\hbox to3em{\hrulefill}\thinspace}
\providecommand{\MR}{\relax\ifhmode\unskip\space\fi MR }
\providecommand{\MRhref}[2]{%
  \href{http://www.ams.org/mathscinet-getitem?mr=#1}{#2}
}
\providecommand{\href}[2]{#2}
\begin{thebibliography}{10}

\bibitem{BHK}
Ph. Briet, H.~Hammedi, and D.~Krej\v{c}i\v{r}\'ik, \emph{Hardy inequalities in
  globally twisted waveguides}, arXiv:1406.2841 [math-SP] (2014).

\bibitem{DE}
P.~Duclos and P.~Exner, \emph{{C}urvature-induced bound states in quantum
  waveguides in two and three dimensions}, Rev. Math. Phys. \textbf{7} (1995),
  73--102.

\bibitem{Edmunds-Evans}
D.~E. Edmunds and W.~D. Evans, \emph{Spectral theory and differential
  operators}, Oxford University Press, 1987.

\bibitem{EKK}
T.~Ekholm, H.~Kova{\v{r}}{\'\i}k, and D.~Krej\v{c}i\v{r}\'{\i}k, \emph{A
  {H}ardy inequality in twisted waveguides}, Arch. Ration. Mech. Anal.
  \textbf{188} (2008), 245--264.

\bibitem{Exner-Barseghyan_2013}
P.~Exner and D.~Barseghyan, \emph{Spectral estimates for {D}irichlet
  {L}aplacians and {S}chr{\"o}dinger operators on geometrically nontrivial
  cusps}, J. Spectr. Theory \textbf{3} (2013), 465--484.

\bibitem{EKov_2005}
P.~Exner and H.~Kova\v{r}{\'\i}k, \emph{Spectrum of the {S}chr{\"o}dinger
  operator in a perturbed periodically twisted tube}, Lett. Math. Phys.
  \textbf{73} (2005), 183--192.

\bibitem{Geisinger-Weidl_2011}
L.~Geisinger and T.~Weidl, \emph{Sharp spectral estimates in domains of
  infinite volume}, Rev. Math. Phys. \textbf{23} (2011), 615--641.

\bibitem{Glazman}
I.~M. Glazman, \emph{Direct methods of qualitative spectral analysis of
  singular differential operators}, Israel Program for Scientific Translations,
  1965.

\bibitem{Haag-Lampart-Teufel_2014}
S.~Haag, J.~Lampart, and S.~Teufel, \emph{Generalised quantum waveguides},
  arXiv:1402.1067 [math-ph] (2014).

\bibitem{K6-with-erratum}
D.~Krej\v{c}i\v{r}\'{\i}k, \emph{Twisting versus bending in quantum
  waveguides}, Analysis on Graphs and its Applications, Cambridge, 2007
  (P.~Exner et~al., ed.), Proc. Sympos. Pure Math., vol.~77, Amer. Math. Soc.,
  Providence, RI, 2008, pp.~617--636. See arXiv:0712.3371v2 [math--ph] (2009)
  for a corrected version.

\bibitem{KKriz}
D.~Krej\v{c}i\v{r}\'{\i}k and J.~K\v{r}\'{\i}\v{z}, \emph{On the spectrum of
  curved quantum waveguides}, Publ.~RIMS, Kyoto University \textbf{41} (2005),
  no.~3, 757--791.

\bibitem{KSed}
D.~Krej\v{c}i\v{r}\'{\i}k and H.~\v{S}ediv\'akov\'a, \emph{The effective
  {H}amiltonian in curved quantum waveguides under mild regularity
  assumptions}, Rev. Math. Phys. \textbf{24} (2012), 1250018.

\bibitem{KZ1}
D.~Krej\v{c}i\v{r}\'{\i}k and E.~Zuazua, \emph{The {H}ardy inequality and the
  heat equation in twisted tubes}, J. Math. Pures Appl. \textbf{94} (2010),
  277--303.

\bibitem{Wachsmuth-Teufel}
J.~Wachsmuth and S.~Teufel, \emph{Effective {H}amiltonians for constrained
  quantum systems}, Mem. Amer. Math. Soc. \textbf{1083} (2013).

\end{thebibliography}
\bibliographystyle{amsplain}
}

\end{document}